\documentclass[12pt]{amsproc}

\usepackage{amsmath, amssymb, amsthm, latexsym,nicefrac,setspace}
\usepackage[pagebackref,colorlinks,linkcolor=red,citecolor=blue,urlcolor=blue,hypertexnames=true]{hyperref}
\usepackage{amsrefs}
\usepackage{marginnote}

\theoremstyle{plain}
\newtheorem{theorem}{Theorem}[section]
\newtheorem{question}[theorem]{Question}
\newtheorem{lemma}[theorem]{Lemma}
\newtheorem{lem}[theorem]{Lemma}
\newtheorem{corollary}[theorem]{Corollary}
\newtheorem{proposition}[theorem]{Proposition}
\newtheorem{definition}[theorem]{Definition}

\theoremstyle{remark}
\newtheorem{remark}[theorem]{Remark}
\newtheorem*{theorem*}{Theorem}

\input{xy}
\xyoption{all}

\def\Z{\mathbb Z}
\def\Q{\mathbb Q}
\def\C{\mathbb{C}}
\def\R{\mathbb{R}}

\def\N{\mathbb{N}}

\newcommand{\iip}[1]{\mathopen{\langle\!\langle}#1\mathclose{\rangle\!\rangle}}

\def\eps{\varepsilon}

\title{Can you compute the operator norm?}

\author{Tobias Fritz}
\address{Tobias Fritz, Perimeter Institute for Theoretical Physics\\ 
31 Caroline St North\\ 
Waterloo, Ontario\\ 
Canada N2L 2Y5}
\email{tobias.fritz@icfo.es}

\author{Tim Netzer}
\address{Tim Netzer, Univ.\ Leipzig,
PF 100920, 04009 Leipzig , Germany}
\email{tim.netzer@math.uni-leipzig.de}

\author{Andreas Thom}
\address{Andreas Thom, Univ.\ Leipzig,
PF 100920, 04009 Leipzig , Germany}
\email{andreas.thom@math.uni-leipzig.de}

\begin{document}

\onehalfspace

\begin{abstract} 
In this note we address various algorithmic problems that arise in the computation of the operator norm in unitary representations of a group on Hilbert space. We show that the operator norm in the universal unitary representation is computable if the group is residually finite-dimensional or amenable with decidable word problem. Moreover, we relate the computability of the operator norm on the group $F_2 \times F_2$ to Kirchberg's QWEP Conjecture, a fundamental open problem in the theory of operator algebras.
\end{abstract}

\maketitle

\tableofcontents

\section{Introduction}

In this article we want to study various algorithmic problems related to the computation of the operator norm on group rings. Let us take some time and state the setup more precisely.
Let $A$ be a finite alphabet, $ F_A$ the free group on the set $A$ and $R \subset  F_A$ be a finite set. We denote by $\Gamma := \langle A | R \rangle$ the group which is generated by the set $A$ subject to relations $R$. The group $\Gamma$ is equipped with a natural surjection $ F_A \to \Gamma$, which we denote by $g \mapsto \bar g$. The kernel of this surjection is $\iip{R}$, the normal subgroup generated by the set $R$. A triple $(\Gamma,A,R)$ as above is called a \emph{presented group}; it is called finitely presented if $R$ is finite. We denote the integral group ring of a group $\Gamma$ by 
$\Z \Gamma$, and the complex group ring by $\C \Gamma$.
For $a \in \Z  F_A$, we denote by $\bar a$ its canonical image in $\Z \Gamma$.

We want to study the operator norm of $\bar a$, considered as an element in the universal group $C^*$-algebra $C^*_{u} \Gamma$ and the reduced group $C^*$-algebra $C^*_{\lambda} \Gamma$. We denote the universal $C^*$-norm on $\C \Gamma$ by $\|.\|_u$ and the operator norm associated with the left-regular representation by $\|.\|_{\lambda}$; more generally, we write $\|.\|_{\varphi}$ for the semi-norm associated to any unitary representation $\varphi \colon \Gamma \to U(H_{\varphi})$:
$$
\|\varphi(a)\|_{\varphi} := \sup \left\{ \|\varphi(a)\xi \| \mid \xi \in H_{\varphi}, \|\xi\| \leq 1 \right\}.
$$ 
For more information about these notions, consult the appendices in \cite{bekval}.
It is well-known that $\|.\|_{\lambda} \leq \|.\|_u$ with equality if and only if $\Gamma$ is amenable; this is Kesten's theorem \cite{kesten}. The natural trace on $\C \Gamma$ is denoted by $\tau \colon \C \Gamma \to \C$, it is given by the formula
$\tau\left(\sum_{g \in \Gamma} a_g g \right) = a_e.$ We denote the cone of hermitian squares in $\C \Gamma$ by $\Sigma^2 \C \Gamma$, that is:
$$\Sigma^2 \C \Gamma := \left\{ \sum_{i=1}^n a_i^*a_i \mid n \in \N, a_i \in \C \Gamma \right\}.$$

We say that a real number $\alpha \in \R$ is \emph {computable} if it can be approximated to any precision with rational numbers by a Turing machine or equivalently, if there is an algorithm which produces two sequences of rational numbers $(p_n)_{n \in \N}$ and $(q_n)_{n \in \N}$, such that $(p_n)_{n \in \N}$ is monotone increasing, $(q_n)_{n \in \N}$ is monotone decreasing and $
\sup_n p_n = \alpha = \inf_n q_n.$ 
Most numbers that we usually think of are computable by their very definition. Even though there are only countably many computable numbers since there are only countably many possible algorithms, one has to think hard to give an explicit example of an uncomputable number. One is Chaitin's constant, whose binary expansion $\sum_{n \in \N} \varepsilon_n 2^{-n}$ encodes the halting problem. Here, $\varepsilon_n \in \{0,1\}$ depending on whether the $n$-th machine in some explicit list of all machines halts. Chaitin's constant is definable in the language of set theory; but not computable. It is important for us that numbers defined in a much simpler language, the first-order language of real closed fields, \emph{are} computable; this is Tarski's famous theorem about quantifier elimination, which we will apply several times. The distinction between computable and definable number goes back to Turing \cite{turing}.

More generally, we want to speak about computable functions depending on an element in the integral group ring of a group.

\begin{definition} \label{comp} Let $(\Gamma,A,R)$ be a presented group.
We say that a function $f \colon \Z \Gamma \to \R$ is computable if there exists an algorithm that takes as input an element $a \in \Z  F_A$ and produces two sequences of rational numbers $(p_n)_{n \in \N}$ and $(q_n)_{n \in \N}$ such that 
\begin{enumerate}
\item $(p_n)_{n \in \N}$ is monotone increasing,
\item $(q_n)_{n \in \N}$ is monotone decreasing, and
\item $\sup_n p_n = f(\bar a) = \inf_n q_n.$
\end{enumerate}
\end{definition}

Obviously, the values of a computable function $f \colon \Z \Gamma \to \R$ are all computable real numbers; the converse does not hold.
Many decision problems in group theory have been studied. The most famous being the \emph{word problem}~\cite{MR0179237,MR0260851,MR0052436}. Here, given a finitely presented group $(\Gamma,A,R)$, the task is to find an algorithm which decides whether an input $g \in  F_A$ satisfies $\bar g = e_{\Gamma}$ or not. 

\begin{remark} \label{comprem}
Note that the computability of the functions $a \mapsto \|a\|_{u}$ or $a \mapsto \|a\|_{\lambda}$ immediately gives a solution to the word problem. Indeed, for $g \in F_A$, we have either $\|\bar g - e_{\Gamma}\| =0$ or $\| \bar g - e_{\Gamma}\| \geq 1$, depending on whether $g$ is trivial in $\Gamma$ or not. Hence, a computation of the operator norm in the sense above gives a decision procedure.
\end{remark}

Our first result is the following converse of the previous remark for the class of amenable groups.

\begin{theorem} \label{mainamenable}
Let $(\Gamma,A,R)$ be a finitely presented amenable group. Then, the word problem for $(\Gamma,A,R)$ is decidable if and only if the function
$$\Z \Gamma \ni a \mapsto \|a\|_{\lambda} \in \R$$
is computable.
\end{theorem}

In general, amenable groups need not have a decidable word problem, as was shown by Kharlampovich \cite{MR631441}. 

Generally speaking: while it is easy, for given $g \in  F_A$, to find a certificate that that $\bar g = e_{\Gamma}$, if that is the case, it is hard (and sometimes impossible) to get a certificate that $\bar g \neq e_{\Gamma}$, if that is the case. In order to provide the second certificate one needs additional information about $\Gamma$.
Let us recall a fundamental result of Mostowskii and McKinsey, and explain its proof since it serves as a motivation for our work.

\begin{theorem}[Mostowskii \cite{most}, McKinsey \cite{mckin}]
\label{RFWP}
Every finitely presented residually finite group $(\Gamma,A,R)$ has a decidable word problem.
\end{theorem}
\begin{proof}
The algorithm does a parallel search for $w \in \iip{R}$ and for finite quotients $\varphi \colon \Gamma \to H$ with $\varphi(w)\neq 0$. The first search is done by enumerating all elements in $\iip{R}$ and comparing them with $w$. The second search is done by enumerating all $A$-tuples of permutations, checking whether they satisfy all relations in $R$, and computing $w$ on the $A$-tuple of permutations. Since $\Gamma$ is residually finite, at some point one of the searches must terminate.
\end{proof}

Therefore, showing undecidability of the word problem is one strategy (maybe not the most promising) for proving that a group is not residually finite. The converse is not true: the Baumslag-Solitar group ${\rm BS}(2,3)$ is not residually finite~\cite{BS}, but has, as a $1$-relator group, decidable word problem~\cite{Magnus}.

In this note we study the property of a group being residually finite-dimensional (RFD), see Definition \ref{rfd}, a certain strenghtening of residual finiteness, and show that it implies computability of the norm in the universal group $C^*$-algebra.
The largest known class of RFD groups contains free groups, Fuchsian groups and many Kleinian groups, see Lubotzky-Shalom \cite{lubsh} and the remarks after Definition \ref{rfd}. For more information on property RFD and related notions consult Brown-Ozawa \cite{brownozawa}. Our main result is the following theorem. 

\begin{theorem}
\label{mainthm}
Let $(\Gamma,A,R)$ be a finitely presented RFD group. Then, $\Gamma$ has a decidable word problem and the function
$$\Z \Gamma \ni a \mapsto \|a\|_u \in \R$$
is computable.
\end{theorem}

The proof of Theorem \ref{mainamenable} and Theorem \ref{mainthm} follow closely the basic idea of the proof of Theorem \ref{RFWP} that we have outlined above. Whereas there is always a sequence of upper bounds, a sequence of lower bounds requires more information about the group and can be provided if the group is RFD or amenable with decidable word problem.

It is a famous open problem in the theory of operator algebras whether the group $ F_2 \times  F_2$ is RFD; nowadays called Kirchberg's Conjecture \cite{Kir}. In principle, a strategy to disprove Kirchberg's Conjecture is to show that the norm in the universal group $C^*$-algebra of $ F_2 \times  F_2$ is in fact not computable. This is not as unreasonable as it may sound, since there are many relatively easy computational problems related to $ F_2 \times  F_2$, which are known to be unsolvable, see the remarks in Section \ref{kirchberg}.

\vspace{0.2cm}

Let us return to the problem of actually computing the operator norm. Once one is in the situation that some number, like the operator norm of some specific element $a \in \Z \Gamma$, is computable, one has to face the following problem. Suppose another computable number $\alpha$ is given in form of a machine that computes it. Can we decide whether $\|a\| = \alpha$? This again is hard, and in general it is impossible to decide if two machines compute the same number. However, we can circumvent this problem in special cases:

\begin{theorem} \label{precthm}Let $ F_A$ be a free group on the finite set $A$. Then there is an algorithm which takes as input $a \in \Z F_A$ and computes a definition of $\|a\|$ in the first order language of real closed fields. In particular, if $\alpha$ is any other number defined in first order language of real closed fields, then it is decidable whether $\|a\| = \alpha$ holds or not.
\end{theorem}

Thanks to Tarski's theorem, the definition of some number $\alpha$ in the first order language of real closed fields takes the form
$$(p(\alpha)=0) \wedge (q_1(\alpha)>0) \wedge \cdots \wedge (q_n(\alpha)>0)$$
for some $n \in \N$ and polynomials $p,q_1,\dots,q_n \in \Z[t]$.

Theorem \ref{precthm} can be used to give algorithms which decide some useful properties. For example, it is decidable  whether $a \in \Z  F_A$ is invertible in $C^*_u( F_A)$ or not. It is not known at the moment whether this result extends to all RFD groups.
There are other variations on Theorem \ref{precthm}. For example we can show, using similar techniques, that there is an algorithm that takes a self-adjoint element $a \in \Z F_A$ as input and produces definitions of real numbers $\mu_1,\mu_2$ such that the spectrum of $a$ in $C^*_u(F_A)$ has the form $[\mu_1,\mu_2] \subset \R$.

\vspace{0.2cm}

The article is organized as follows. In Section \ref{secuniv} we study the norm in the universal representation and prove Theorem \ref{mainthm} and Theorem \ref{precthm}. Theorem \ref{mainamenable} is proved in Section \ref{thmamen}. In Section \ref{kirchberg}, we discuss a relation with Kirchberg's QWEP Conjecture and speculate about a relationship with some algorithmic problems related to the group $F_2 \times F_2$, which are known to be undecidable.

\section{Computability of the norm in the universal representation} \label{secuniv}

The following lemma provides the key to the computation of a sequence of upper bounds. It can be regarded as a special case of a strict Positivstellensatz due to Schm\"udgen~\cite{Schm}.

\begin{lemma} \label{univ} Let $a \in \Z \Gamma$.
\label{semdef}
$$
\|a\|_u = \inf \left\{ \lambda \in \R_{\geq 0} \mid \lambda^2 - a^*a \in \Sigma^2 \C \Gamma \right\}
$$
\end{lemma}

\begin{proof}
Clearly $\Sigma^2 \C \Gamma \subseteq \left(C_u^*\Gamma\right)_+:=\{x^*x\mid x\in C_u^*\Gamma\}$. Conversely, we claim that $x + \varepsilon1 \in \Sigma^2 \C \Gamma$ for every $x\in \left(C_u^*\Gamma\right)_+\cap \C\Gamma$ and $\varepsilon > 0$. For if this were not the case for some $x$, then the Riesz extension theorem~\cite{Akh} would guarantee the existence of a linear map $\varphi : \C \Gamma \to \C$ with $\varphi \left( \Sigma^2 \C \Gamma \right) \subseteq \R_+$ and $\varphi(x) < 0$. For this it is essential that $1$ is an algebraic interior point in $\Sigma^2\C\Gamma$, as shown in \cite{cimpric}. The GNS construction turns this $\varphi$ into a unitary representation $\pi_{\varphi}$ with the property that $\pi_{\varphi}(x) \not\geq 0$, so that $x \not \in \left(C_u^*\Gamma\right)_+$.
\end{proof}

If $(\Gamma,A,R)$ is a finitely presented group, then one can find a convergent sequence of upper bounds on $\|a\|_u$ as follows. Let $F_{A,n}$ be the set of elements in $F_A$ with word length less or equal $n$. Let us write $Q(A,R)$ for the quadratic module in $\C F_A$ generated by $\{1-r \mid r \in R\}$. More precisely, we set

$$Q(A,R) := \left\{  \sum_{r \in R \cup \{0\}}\sum_{k = 1}^{n_r} b_{r,k}^*(1-r)b_{r,k} \:\Bigg|\: n_i \in \N, b_{i,k} \in \C F_A, \forall i \in R \cup \{0\} \right\}.$$
Obviously, $Q(A,R)$ is a convex cone and functionals on $\C F_A$ which are positive on $Q(A,R)$ are in bijection with positive functionals on $\C \Gamma$. In particular, the proof of Lemma \ref{univ} yields 
\begin{equation} \label{univg}
\|\bar a\|_u = \inf \left\{ \lambda \in \R_{\geq 0} \mid \lambda^2 - a^*a \in Q(A,R) \right\}
\end{equation}
 for all $a \in \Z F_A$. Here, $\bar a$ denotes the canonical image of $a$ in $\Z \Gamma$.

We define $Q_n(A,R)$ to be the subset of those elements in $Q(A,R)$, which have a representation with all $b_{i,k} \in \C F_{A,n}$. 
Then $Q_n(A,R)$ is finite-dimensional, and $\bigcup_n Q_n(A,R) = Q(A,R)$.  In fact,
$$
Q_n(A,R) = \left\{ \sum_{r \in R \cup \{0\}}\sum_{g,h \in F_{A,n}} C_{r,g,h} g^{-1}(1-r)h \:\Bigg|\: (C_{r,g,h})_{g,h \in \Gamma_n}\in  \bigoplus_{r \in R \cup \{0\}}M_{F_{A,n}}(\C)_+ \right\} .
$$
For $a\in\Z F_A$, we consider all $n\geq n_0$ where $n_0$ is such that $a\in\C F_{A,n_0}$ and $\Lambda - a^*a \in Q_{n_0}(A,R)$ for some $\Lambda\in\R$. Then by Equation \eqref{univg}, 
\begin{equation}
\label{semdefprimal}
\|a\|_u^2 \leq \min \left\{ \Lambda \in \R \:\Bigg|\: 
\begin{array}{c}
\exists (C_{r,g,h})_{g,h \in \Gamma_n}\in \bigoplus_{r \in R \cup \{0\}}M_{F_{A,n}}(\C)_+ \\
\textrm{ with } \Lambda - a^*a = \sum_{r \in R \cup \{0\}}\sum_{g,h} C_{g,h} g^{-1} (1-r)h 
\end{array}
\right\} , 
\end{equation}
where the right-hand side is now just a semidefinite programming problem in matrices $\bigoplus_{r \in R \cup \{0\}} M_{F_{A,n}}(\C).$ As shown in Equation \eqref{univg}, this bound becomes tight for $n\to\infty$. So, computing the value of the semidefinite program by bounding it from above with an accuracy of, say $1/n$, provides a convergent sequence of upper bounds on $\|a\|_u$. For more details on semidefinite programming see for example \cite{wo}. This shows the following:

\begin{corollary} \label{corup}
For any finitely presented group $(\Gamma,A,R)$ there is an algorithm computing a convergent sequence of upper bounds on $\|.\|_u$ on $\Z\Gamma$.
\end{corollary}

Semidefinite programming duality provides another point of view on~(\ref{semdefprimal}). We claim that
\begin{equation}
\label{semdefdual}
\|a\|_u^2 \leq \max \left\{ \varphi(a^*a) \::\: \varphi : \C F_{A} \to \C, \: \varphi\left(Q_n(A,R)\right)\subseteq \R_{\geq 0}, \:\varphi(1)=1 \right\} ,
\end{equation}
where the right-hand side is a semidefinite program dual to~(\ref{semdefprimal}) for $n\geq n_0$, and the two optimal values coincide. To see this, note first that both semidefinite programs are feasible:~(\ref{semdefprimal}) is feasible by the assumption $n\geq n_0$, the second one is feasible since $(-1)\not\in Q_{n}(A,R)$ and the Hahn-Banach theorem. Therefore, for any feasible solutions $(C_{r,g,h})_{g,h}$ and $\varphi$,
\begin{equation}
\label{duality}
\varphi(a^*a) \leq \Lambda - \varphi\left(\sum_{r \in R \cup \{0\}}\sum_{g,h} C_{r,g,h} g^{-1} h\right) \leq \Lambda.
\end{equation}
We claim that there are $\varphi$ and $(C_{r,g,h})_{g,h}$ for which this bound is tight. This is so since $\Lambda - a^*a$ lies, for the optimal $\Lambda$, on the boundary of the cone $Q_n(A,R)$, so that there exists a functional $\varphi : \C F_{A,n} \to \C$ with $\varphi(\Lambda - a^*a) = 0$. This $\varphi$ is optimal since $\varphi(a^*a) = \Lambda$, which saturates~(\ref{duality}).

We end the discussion of semidefinite programming by noting that these ideas not only apply to the universal $C^*$-norm on group rings, but on any $*$-algebra with a finite presentation in terms of generators and equality or positivity relations for linear combinations of words in those generators. This has been worked out in more detail in~\cite{NPA2} as ``noncommutative polynomial optimization'' (note that the ``primal'' and ``dual'' conventions of~\cite{NPA2} are opposite to ours). Many applications and particular cases had already been studied earlier; this includes hierarchies of semidefinite programs for commutative moment problems~\cite{Lasserre} and noncommutative moment problems arising in quantum information theory~\cite{NPA}.

If there was any way to understand efficiently how fast the sequence of upper bounds converges, then one could turn Corollary \ref{corup} into an actual computation of the operator norm in the sense of Definition \ref{comp}. However, this seems to be out of reach even for reasonable groups and is impossible in general, as it would imply decidability of the word problem by Remark~\ref{comprem}. This is in contrast to the commutative case, for which convergence bounds have been derived~\cite{BH,DW}.
In order to provide interesting lower bounds on the norm in the universal representation, we have to make additional assumptions on $\Gamma$.

In the following, we need basic properties of the unitary dual of a discrete group. For details about the unitary dual and the Fell toplogy on it consult the informative appendices in \cite{bekval}.

\begin{lemma} Let $\Gamma$ be a group and $\Phi$ a set of unitary representations which is dense in the unitary dual of $\Gamma$. Then for every $a\in\Z\Gamma$,
$$
\|a\|_u = \sup \left\{ \|\varphi(a)\|_{\varphi} \mid \varphi \in \Phi  \right\}.
$$
\end{lemma}

\begin{proof}
This is well-known and an immediate consequence of the definition of the Fell topology on the unitary dual.
\end{proof}

Let us now give a definition of property RFD.

\begin{definition} \label{rfd}
A group $\Gamma$ is called \emph{residually finite-dimensional (RFD)} if the set of finite-dimensional unitary representations is dense in the unitary dual of $\Gamma$.
\end{definition}

Note that if $\Gamma$ is finitely generated and RFD, then finite-dimensional representations must separate the elements of $\Gamma$, and $\Gamma$ follows to be residually finite by Mal'cev's theorem~\cite{Malcev}.

Finitely generated Fuchsian groups and fundamental groups of closed hyperbolic 3-manifolds which fiber over the circle are known to be RFD, see \cite{lubsh}.
Indeed, it is well-known that free groups are RFD, see for example Theorem 2.2 in \cite{lubsh}. Theorem 2.8 in \cite{lubsh} shows that surface groups and fundamental groups of closed hyperbolic 3-manifolds that fiber over the circle have RFD. Now, it is easy to see that RFD passes to finite index extensions. This implies the claim for Fuchsian groups since every Fuchsian group contains a free group or a surface group with finite index. That every fundamental group of a closed hyperbolic 3-manifold admits a subgroup of finite index which fibers over the circle is known as ThurstonÕs Virtual Fibration Conjecture.

We write $f$ for the direct sum of all finite-dimensional unitary representations of $\Gamma$. The following observation can be regarded as an alternative definition of RFD:

\begin{lem}
$\Gamma$ is RFD if and only if $\|.\|_u = \|.\|_{f}$ on $\Z\Gamma$.
\end{lem}

\begin{proof}
Again by definition of the Fell topology, we know that $\|.\|_u=\|.\|_f$ on $\C\Gamma$ if and only if $\Gamma$ is RFD. By homogeneity of norms and density of the inclusion $\Q\subseteq\R$, the assumption $\|.\|_u = \|.\|_f$ on $\Z\Gamma$ implies that this equality also holds on $\R\Gamma$, and the problem is to show that this implies the equality on all of $\C\Gamma$.

We decompose $\C\Gamma = \R\Gamma \oplus i\cdot \R\Gamma$. By assumption, the two norms coincide on each of the two summands; moreover, for $a,b\in\R\Gamma$,
$$
\|a+ib\|_u = \|a-ib\|_u ,\qquad \|a+ib\|_f = \|a-ib\|_f
$$
since every representation of $\Gamma$ has a complex conjugate, and taking complex conjugates preserves finite-dimensionality. 
Now it follows from elementary estimates like~\cite{Fri1}*{Prop.~5.6} that $\|.\|_u$ and $\|.\|_f$ differ on $\C\Gamma$ at most by a factor of $2$. So,
$$
\|.\|_f \leq \|.\|_u \leq 2 \|.\|_f .
$$
In particular, the $C^*$-completions $C^*_u\Gamma$ and $C^*_f\Gamma$ are canonically isomorphic and the canonical surjection
$$
\varphi : C^*_u\Gamma \twoheadrightarrow C^*_f\Gamma.
$$
is an isomorphism. Now the assertion follows from the uniqueness of the norm on a $C^*$-algebra.
\end{proof}

We are now ready to prove Theorem \ref{mainthm}.

\begin{proof}[Proof of Theorem \ref{mainthm}:]
The first assertion is clear by Theorem~\ref{RFWP} since the assumption implies that $\Gamma$ is residually finite.

In view of Corollary \ref{corup} it remains to provide a convergent sequence of lower bounds on the operator norm in the universal representation. Let $a \in \Z  F_A$. Let $n \in \N$ and consider the set 
$$X(n) = \left\{ (u_a)_{a \in A} \subset U(n)^A \mid r((u_a)_{a \in A})=1 \ \forall r \in R \right\} \subseteq U(n)^A.$$
Clearly, $X(n)$ is a compact real algebraic subset of $\R^{|A| \cdot 2n^2}$. Denote by $D(n) = \{ \xi \in \C^n \mid \|\xi\| \leq 1 \}$ and the function
$$f_n \colon X(n) \times D(n) \to \R, \quad f_n((u_a)_{a \in A},\xi) := \|a((u_a)_{a \in A}) \xi\|^2.$$
Denote the maximum of $f_n$ on $X(n) \times D(n)$ by $\alpha_n$.
By the previous lemmas and the assumptions on $\Gamma$, we have
$$\|\bar a\|_u = \sup \left\{ \alpha_n^{1/2} \mid n \in \N \right\}.$$
Thanks to Tarski's real quantifier elimination~\cite{Tarski}, each $\alpha_n$ is computable, and therefore $(\alpha_n^{1/2})_{n\in\N}$ is the required sequence of lower bounds. This proves the claim.
\end{proof}

Again, note that this theorem and its proof directly generalize from group $*$-algebras to arbitrary finitely presented $*$-algebras.

\begin{question}
The group $\Gamma=SL_3(\Z)$ is known not to be RFD, see \cite{bekka}. Is the function
$$\Z \Gamma \ni a \mapsto \|a\|_u \in \R$$
computable?
\end{question}

We now turn to the proof of Theorem \ref{precthm}. We need the following lemma.

\begin{lemma} \label{choi}
Let $a \in \Z F_A,$ $n=\vert A\vert$, and let $d \in \N$ be the length of the longest word appearing in the support of $a$. There exists a unitary representation $\pi \colon F_A \to U(k)$ of dimension $k:=4n^{d}$ and a unit vector $\xi \in \C^k$ such that
$\|a\|_u = \|\pi(a) \xi\|_{\pi}.$
\end{lemma}
\begin{proof}
The proof is an application of what is known as Choi's trick \cite[Theorem 7]{choi} which gives an efficient proof that the group $\Gamma=F_A$ is RFD. Let $a \in \Z \Gamma$. It is a standard consequence of the compactness of the state space of $C^*_u\Gamma$ that the operator norm $\|a\|_u$ is achieved at some vector in some unitary representation.
Let $\sigma \colon F_A \to U(H_{\sigma})$ be a unitary representation and $\xi' \in H_{\sigma}$ be a unit vector with $\|a\|_u = \|\sigma(a) \xi'\|_{\sigma}$. Consider $H$, the linear span of $\sigma(g)\xi$ for all $g$ with length less or equal $d$. The dimension of $H$ is at most $2n^{d}$. Let $p$ be the orthogonal projection from $H_{\sigma}$ onto $H$ and denote the generators of $F_A$ by $v_1,\dots,v_n$.

We set
$$u_i:= \left(\begin{array}{cc}p\sigma(v_i)p & \sqrt{1_H-p\sigma(v_i)p \sigma(v_i)^*p} \\ \sqrt{1_H-p\sigma(v_i)^*p\sigma(v_i)p} & -p\sigma(v_i)^*p\end{array}\right)\in\mathcal{L}(H\oplus H).$$
It is easy to check that $u_1,\dots,u_n$ are unitary and we let $\pi$ be the unitary representation on $H \oplus H$ associated with them. Again, it is easy to check that
$\pi(a)(\xi',0)^t = (\sigma(a)\xi,0)$ and hence, $\|a\|_u = \|\pi(a)\xi\|_{\pi}$ for the unit vector $\xi = (\xi',0)^t$. This finishes the proof.
\end{proof}

We are now ready to prove Theorem \ref{precthm}.
\begin{proof}[Proof of Theorem \ref{precthm}:] Using the notation of the proof of Theorem \ref{mainthm}, Lemma \ref{choi} gives that $\|a\|_u = \alpha_{k}$ for some sufficiently large and computable integer $k$. Clearly, $\alpha_k$ is defined in the first order language of real closed fields. Moreover, the equation $\alpha = \alpha_k$ is decidable if $\alpha$ defined in the same language, again thanks to Tarski's theorem.
\end{proof}

As a corollary to Theorem \ref{precthm} we can solve algorithmic problems which require precise information about the spectrum of some element in the integral group ring.

\begin{corollary} There is an algorithm that takes as input $a \in \Z F_A$ and decides whether $a$ is invertible in the universal group $C^*$-algebra.
\end{corollary}
\begin{proof}
For $a = \sum_g a_g g \in \Z F_A$, we write $\|a\|_1 := \sum_g |a_g|$. It is clear that $\|a\|_u \leq \|a\|_1$ for all $a \in \Z F_A$.
Now, the element $a \in \Z F_A$ is invertible in $C_u^* F_A$ if and only if $a^*a \in \Z F_A$ is invertible. Let $\Lambda \in \Z$ be a computable upper bound for $\|a^*a\|_u$ such as $\|a\|^2_1$. Then the spectral theorem implies that
$\|\Lambda - a^*a\|_u = \Lambda$ if and only if $a^*a$ is not invertible in $C^*_u F_A$. This proves the claim, since Theorem \ref{precthm} provides a decision procedure for this equality.
\end{proof}

\section{Lower bounds on the norm in the left regular representation}
\label{thmamen}

In this section we provide a convergent sequence of lower bounds on the norm in the left-regular representation. In the case of amenable groups, this leads to a computation of the natural norm on the integral group ring, using Corollary \ref{corup} Êand the fact that $\Vert . \Vert_\lambda=\Vert . \Vert_u$. Note that formulas for norms in left regular representations for certain classes of groups and elements have already been obtained in  \cite{ake, leh} for example.

Recall that the group ring $\Z\Gamma$ naturally commes equipped with the trace
$$
\tau \colon\Z\Gamma \to \Z ,\qquad \sum_g a_g g \mapsto a _e, 
$$
which extends uniquely to a tracial state on $C^*_{\lambda}\Gamma$. This trace is faithful, i.e.\ we have that $\tau(a^*a)=0$ implies $a=0$. the following result is well-known and we include a proof for convenience.

\begin{lemma} \label{amenable} Let $a \in \Z \Gamma$. Then,
$$
\|a\|_{\lambda} = \sup \left\{ \tau(b^*a^*ab)^{1/2} \mid b \in \C \Gamma, \tau(b^*b) \leq 1 \right\} = \sup \left\{ \tau((a^*a)^n)^{1/2n} \mid n \in \N \right\} .
$$
\end{lemma}

\begin{proof}
By definition, $\ell_2(\Gamma)$ is the completion of $\C\Gamma$ with respect to the norm $b\mapsto \tau(b^*b)^{1/2}$, and it carries a natural action of $\Gamma$ by left multiplication. This directly implies the first equation.

We now consider the second equation. The inequality $\tau((a^*a)^n) \leq \|a^*a\|_{\lambda}^n\leq \|a\|_{\lambda}^{2n}$ is clear, so that the main task is proving the other direction. We work in the group von Neumann algebra ${\mathcal N}\Gamma$. Fix any $\eps>0$ and consider the spectral projection $p$ defined by applying Borel functional calculus to $a^*a$ with respect to the indicator function of the interval $[ \|a^*a\|_{\lambda} - \eps, \|a^*a\|_{\lambda} ]$. Then $a^*a \geq \left(\|a^*a\|_{\lambda} - \eps\right) p$, so that
$$
\tau((a^*a)^n) \geq \tau((\|a^*a\|_{\lambda} - \eps)^n p) = \left(\|a^*a\|_{\lambda} - \eps\right)^n \tau(p) .
$$
Faithfulness of $\tau$ together with $p\neq 0$ implies $\tau(p)>0$, so that the right-hand side of
$$
\tau((a^*a)^n)^{1/2n} \geq (\|a^*a\|_{\lambda} - \eps)^{1/2} \tau(p)^{1/2n}
$$
tends to $(\|a^*a\|_{\lambda}-\eps)^{1/2}$ as $n\to\infty$. The conclusion follows since $\eps$ was arbitrary.
\end{proof}

Note that the proof, and therefore the lemma, apply similarly to any von Neumann algebra equipped with a faithful tracial state $\tau$.
We are now ready to prove Theorem \ref{mainamenable}.

\begin{proof}[Proof of Theorem \ref{mainamenable}:]
Concerning $(1) \Rightarrow (2)$, Lemma~\ref{amenable} provides a convergent sequence of computable lower bounds and Lemma~\ref{univ} a convergent sequence of computable upper bounds. $(2) \Rightarrow (1)$ was essentially answered in Remark \ref{comprem}: We are able to algorithmically decide for $g \in  F_A$ whether $\|1- \bar g\|_{\lambda}=0$ or $\|1- \bar g\|_{\lambda}\geq 1$, and it is easy to see that one of the two cases must occur.
\end{proof}

%
%
%
%

\section{Relation to the Kirchberg's QWEP Conjecture} \label{kirchberg}

Let us finish this note with a question and a relation to some famous open problems in the theory of operator algebras.

\begin{question} \label{product}
Consider $A= \{x,y,z,w\}$ and 
$$\Gamma =  F_2 \times  F_2 = \langle x,y,z,w \mid [x,z], [x,w], [y,z], [y,w] \rangle.$$ Is the function $\Z \Gamma \ni a\mapsto \| a\|_{u} \in \R$ computable?
\end{question}

Kirchberg's seminal work~\cite{Kir} shows that a positive answer to the famous Connes Embedding Problem is equivalent to $ F_2 \times  F_2$ being RFD. The question whether $F_2 \times F_2$ (or equivalently $F_n \times F_n$ for any $n \geq 2$) is RFD is generally known as a version of Kirchberg's QWEP Conjecture. Consult \cite{MR2072092} and the references therein for more details about these fundamental conjectures. By Theorem~\ref{mainthm}, a positive solution to any of these conjectures would also imply a positive answer to Question~\ref{product}. 
This elucidates the importance of the computability of $\|.\|_u$ on the group ring $\Z(F_n\times F_n)$: if this norm is not computable, this would refute Connes Embedding Problem and Kirchberg's QWEP conjecture. On the other side, if this norm is computable, and its computation would even turn out to be practical, this would be very interesting for applications in quantum information theory concerning the maximal quantum violations of Bell inequalities~\cite{Fri2}.

\vspace{0.2cm}

We want to end by recalling Miha\u{\i}lova's construction from \cite{MR0222179} which shows that rather reasonable algorithmic questions, which are known to be solvable for free groups and surface groups, become intractable for products of free groups.
Our initial hope was that we might be able to relate the computability of $\|.\|_u$ on $\Z(F_n\times F_n)$ to the decidability of some of these decision problems for $F_n\times F_n$, which have been studied so extensively. In particular, the hope was to relate it to the \emph{membership problem} for finitely generated subgroups of $F_n \times F_n$. In general, for a subgroup $\Lambda \subset \Gamma$, the membership problem takes some $g\in \Gamma$ as input and asks us to decide whether $g\in \Lambda$ or $g\not\in \Lambda$; typically one assumes here that the word problem for $\Gamma$ is decidable and $\Lambda$ is finitely generated by some finite set which is part of the input.
Miha\u{\i}lova's construction \cite{MR0222179} provides examples of finitely generated subgroups of $F_n \times F_n$, for which this problem is undecidable.
Let $\Gamma = \langle g_1,\dots,g_n |R \rangle$ be a finitely presented group, let $\pi \colon F_n \to \Gamma$ be the natural surjection, and consider the kernel pair
\begin{equation}
\label{Msubgroup}
\Lambda := \{ (g,h) \in F_n \times F_n \mid \pi(g)=\pi(h) \}.
\end{equation}
It is easy to see that the subgroup $\Lambda\subset F_n\times F_n$ is generated by $R \times \{e\}$ and a diagonal copy of $F_n$. Hence, $\Lambda \subset F_n \times F_n$ is finitely generated by some explicit set of generators. Moreover, the word problem for the finitely presented group $\Gamma$ is equivalent to the membership problem for the inclusion $\Lambda \subset F_2 \times F_2$. Indeed, $(v,w) \in \Lambda$ if and only if $\pi(v^{-1}w)=e$ in $\Gamma$. Combining this with the existence of finitely presented groups with undecidable word problem \cite{MR0092784,BS,MR631441}, Miha\u{\i}lova showed:

\begin{proposition}[Miha\u{\i}lova] The membership problem for finitely generated subgroups of $F_n \times F_n$ is not decidable for any $n \geq 2$.
\end{proposition}

We think that this result supports the point of view that a positive answer to Question \ref{product} is too much to hope for.

\vspace{0.5cm}

\begin{center}
\bf Acknowledgments
\end{center}

\vspace{0.5cm}

The research of A.T.\ was supported by ERC. Research at Perimeter Institute is supported by the Government of Canada through Industry Canada and by the Province of Ontario through the Ministry of Economic Development and Innovation. Part of this research was done when T.F.\ visited University of Leipzig in 2012. We thank the unknown referee for a helpful comment that improved the statement of Corollary \ref{corup}.

\end{document}